\documentclass[11pt]{amsart}
\usepackage{pdfsync, xypic}

\usepackage{amsmath, amssymb}

\renewcommand{\Re}{Re}

\DeclareMathOperator{\Ext}{Ext}
\DeclareMathOperator{\dist}{dist}

\DeclareMathOperator{\Ann}{Ann}

\newcommand{\bbN}{{\mathbb N}}

\newcommand{\bbR}{\mathbb R}
\newcommand{\bbC}{\mathbb C}

\newcommand{\cZ}{{\mathcal Z}}

\newcommand{\bt}{\mathbf t}

\newcommand{\cC}{\mathcal C}

\newcommand{\cB}{\mathcal B}

\newcommand{\e}{\varepsilon}

\newtheorem{thm}{Theorem}
\newtheorem{theorem}{Theorem}[section] 
\newtheorem{corollary}[theorem]{Corollary}

\newtheorem{remark}[theorem]{Remark}

\newtheorem{question}[theorem]{Question}

\newtheorem{lemma}[theorem]{Lemma}
\newtheorem{prop}[theorem]{Proposition}

\theoremstyle{definition}

\newtheorem{definition}[theorem]{Definition}

\newtheorem{problem}[theorem]{Problem}

\DeclareMathOperator{\Fin}{Fin}

\newcounter{my_enumerate_counter}
\newcommand{\pushcounter}{\setcounter{my_enumerate_counter}{\value{enumi}}}
\newcommand{\popcounter}{\setcounter{enumi}{\value{my_enumerate_counter}}}

\newcommand{\cP}{{\mathcal P}}

\newcommand{\cU}{{\mathcal U}}

\DeclareMathOperator{\Ad}{Ad}

\title{Countable saturation of corona algebras}

\author{Ilijas Farah}

\address{Department of Mathematics and Statistics\\
York University\\
4700 Keele Street\\
North York, Ontario\\ Canada, M3J 1P3\\
and Matematicki Institut, Kneza Mihaila 34, Belgrade, Serbia}

\urladdr{http://www.math.yorku.ca/$\sim$ifarah}

\email{ifarah@mathstat.yorku.ca}
\author{Bradd Hart} 
\address{Dept. of Mathematics and Statistics\\
McMaster University\\ 1280 Main Street\\ West Hamilton, Ontario\\
Canada L8S 4K1}

\email{hartb@mcmaster.ca}

\urladdr{http://www.math.mcmaster.ca/$\sim$bradd/}

\thanks{Partially supported by NSERC. The first author would like  to thank Saeed Ghasemi and 
Paul McKenney for pointing out to several typos in the earlier version and to Mikael R\o rdam for a helpful remark. }
   
\date{\today}
\begin{document}
\begin{abstract} 
We present  unified proofs of several properties of the corona of 
$\sigma$-unital C*-algebras such as AA-CRISP, SAW*, being sub-$\sigma$-Stonean in the 
sense of Kirchberg, and 
the conclusion of Kasparov's Technical Theorem. Although 
our results
were obtained by considering C*-algebras as models of the logic for metric structures, 
the reader is not required to have any knowledge
of model theory of metric structures (or model theory, or logic in general). 
The proofs involve analysis of the extent of model-theoretic saturation of  corona algebras.  \\

\noindent {\sc R\'esum\'e.}
Nous pr\'esentons des d\'emonstrations unifi\'ees de plusieurs propri\'et\'es de la corona des C*-algebras $\sigma$-unitales tel qu'AA-CRISP, SAW*, \'etant sous-$\sigma$-Stonean dans le sens de Kirchberg, et la conclusion du th\'eor\`eme technique de Kasparov. Bien que nos r\'esultats aient \'et\'e obtenus en consid\'erant les C*-algebras comme mod\`eles de la logique pour les structures m\'etriques, le lecteur n'est pas requis d'avoir aucune connaissance de la th\'eorie des mod\`eles des structures m\'etriques (ou la th\'eorie des mod\`eles, ou de la logique en g\'en\'eral). Les 
d\'emonstrations impliquent l'analyse de l'ampleur de la saturation mod\`ele-th\'eor\'etique des algebres de corona.
\end{abstract} 

\maketitle
We shall investigate the degree of countable saturation of  coronas (see Definition~\ref{D.Deg1} 
and paragraph following~it). This  property is shared by 
ultraproducts associated with nonprincipal ultrafilers on $\bbN$ in its full form. 
The following summarizes our main results. All ultrafilters are nonprincipal ultrafilters on~$\bbN$.  

\begin{thm} \label{T0} Assume a C*-algebra $M$ is in one of the following forms: 
\begin{enumerate}
\item\label{I.T-1.1}  the corona of a $\sigma$-unital C*-algebra, 
\item \label{I.T-1.2} an ultraproduct of a sequence of C*-algebras, 
\item \label{I.T-1.3} an ultrapower of a C*-algebra, 
\item \label{I.T-1.4} $\prod_n A_n/\bigoplus_n A_n$, for unital C*-algebras $A_n$,  
\item \label{I.T-1.5} the relative commutant of a separable subalgebra of an algebra that is in one of the forms
\eqref{I.T-1.1}--\eqref{I.T-1.4}. 
\pushcounter
\end{enumerate}
Then $M$ satisfies each of the following (see below for definitions): 
\begin{enumerate}
\popcounter
\item\label{J.T-1.1}  It is  SAW*
\item\label{J.T-1.2}  It  has AA-CRISP (asymptotically abelian, countable Riesz separation property), 
\item\label{J.T-1.3}  The conclusion of Kasparov's technical theorem, 
\item\label{J.T-1.4}  It is $\sigma$-sub-Stonean in the sense of Kirchberg, 
\item \label{J.T-1.5} Every derivation of a separable subalgebra of $M$ is of the 
form $\delta_b$ for some $b\in M$. 
\end{enumerate}
\end{thm}

\begin{proof} 
Each of these classes of C*-algebras is countably degree-1 saturated (Definition~\ref{D.Deg1}). 
For \eqref{I.T-1.1} this is  Theorem~\ref{T1}, proved in \S\ref{S.Proofs}. 
For \eqref{I.T-1.2} and \eqref{I.T-1.3} this is a consequence of \L os's theorem 
(see e.g., \cite[Proposition~4.11]{FaHaSh:Model2}). 
Every algebra as in \eqref{I.T-1.4} is the corona of $\bigoplus_n A_n$ so this is a special case
of \eqref{I.T-1.1}. 
For~\eqref{I.T-1.5} this is Lemma~\ref{L.cta}. 

Property \eqref{J.T-1.1} now follows by Proposition~\ref{P.SAW}, 
\eqref{J.T-1.2} follows by Proposition~\ref{P.AA},  
 \eqref{J.T-1.3} follows by Proposition~\ref{P.KTT}, 
 \eqref{J.T-1.4} follows by Proposition~\ref{P.ssS}, 
 and \eqref{J.T-1.5} follows by Proposition~\ref{P.derivation}. 
\end{proof} 

The assertion `every approximately inner automorphism of 
a separable subalgebra of $M$ is implemented by a unitary in $M$' 
is true for algebras as in \eqref{I.T-1.2}, \eqref{I.T-1.3} or the corresponding instance of  \eqref{I.T-1.4}
(Lemma~\ref{L.auto.1}). 
However this is not true in case when $M$ is 
the Calkin algebra (see Proposition~\ref{P2}).

By \cite{Gha:SAW*} no SAW*-algebra can be written as a tensor product of two 
infinite-dimensional C*-algebras. By Theorem~\ref{T0}, this applies to every 
  C*-algebra $M$ satisfying any of \eqref{I.T-1.1}--\eqref{I.T-1.5}.

\subsection*{Organization of the paper} In \S\ref{S.Intro} we introduce terminology and 
state the main results. Applications are given in \S\ref{S.App}, and proofs of main results
are in \S\ref{S.Proofs}. In \S\ref{S.Limiting} we demonstrate that degree of saturation of
the Calkin algebra is rather mild. In \S\ref{S.Concluding} we list several open problems.

\section{Introduction} \label{S.Intro}
For $F\subseteq \bbR$ and $\e>0$ we write $F_{\e}=\{x\in \bbR: \dist(x,F)\leq \e\}$. 
Given a C*-algebra $A$, a \emph{degree 1 *-polynomial} in variables $x_j$, for $j\in \bbN$, 
with coefficients in $A$
is a linear combination of terms of the form $ax_j b$, $a x_j^* b$ and $a$ with $a,b$ in $A$. 
We write $M_{\leq 1}$ for the unit ball of a C*-algebra $M$. 
 
\begin{definition} \label{D.Deg1} 
A metric structure $M$ is \emph{countably degree-1  saturated} if 
for every countable family of degree-1 *-polynomials $P_n(\bar x)$ with coefficients 
in~$M$ and variables $x_n$, for $n\in \bbN$, and every family of compact 
sets $K_n\subseteq \bbR$, for $n\in \bbN$,  the following are equivalent. 
\begin{enumerate} 
\item\label{I.type.1} There are $b_n\in M_{\leq 1}$, for $n\in \bbN$, such that $P_n(\bar b)\in K_n$ for all $n$. 
\item\label{I.type.2} For every $m\in \bbN$ there are $b_n\in M_{\leq 1}$, for $n\in \bbN$, 
such that $P_n(\bar b)\in (K_n)_{1/m}$ for all $n\leq m$. 
\end{enumerate}

More generally, if $\Phi$ is a class of *-polynomials, we say that 
$M$ is   \emph{countably $\Phi$-saturated} if 
for every countable family of  *-polynomials $P_n(\bar x)$ in $\Phi$ with coefficients 
in $M$ and variables $x_n$, for $n\in \bbN$, and every family of compact 
sets $K_n\subseteq \bbR$, for $n\in \bbN$ the assertions  \eqref{I.type.1} and \eqref{I.type.2} above are equivalent. 

If $\Phi$ is the class of all *-polynomials then instead of $\Phi$-saturated we say \emph{countably quantifier-free saturated}. 
\end{definition}

Note that 
by compactness we obtain an equivalent definition if we require each $K_n$ to be a singleton. 

With the obvious definition of `degree-$n$ saturated' one might expect to have a proper hierarchy of 
levels of saturation. However, this is not the case. 
 
\begin{lemma} An algebra that is degree-2 saturated is necessarily quantifier-free saturated. 
\end{lemma} 

\begin{proof} Assume $C$ is degree-2 saturated and $\bt$ is a consistent 
countable quanti\-fier-free type over $C$. 
By compactness and  Stone--Weierstrass approximation theorem we may assume that $\bt$ consists of 
formulas of the form $\|P(\bar x)\|=r$ for a polynomial $P$. 
By adding a countable set of new variables $\{z_i\}$ and formulas $\|xy-z_i\|=0$ for distinct variables $x$ and $y$ occurring in $\bt$, one can reduce the degree of all polynomials occurring in $\bt$. By 
repeating this procedure countably many times one obtains a new type $\bt'$ in countably many variables
such that $\bt'$ does not contain polynomials of degree higher than 2, it is consistent, and a realization
of $\bt'$ gives a realization of $\bt$. 
\end{proof} 

In the following it is assumed that each $P_n$ is a *-polynomial with coefficients in $M$, 
and reference to the ambient algebra $M$ is omitted whenever it is clear from the context. 
An expression of the form $P_n(\bar x)\in K_n$ is called a \emph{condition} (over $M$). 
A set of conditions is a \emph{type} (over $M$). If all conditions involve only polynomials in $\Phi$
then we say that type is a \emph{$\Phi$-type}. 
If all coefficients of polynomials occurring in type $\bt$ belong to a set $X\subseteq M$ 
then we say $\bt$ is 
\emph{type over $X$}. 
A type satisfying \eqref{I.type.2} is \emph{approximately finitely satisfiable} (in $M$), or more succinctly 
\emph{consistent} with $M$, 
and a type satisfying \eqref{I.type.1} is \emph{realized} (in $M$) by $\bar b$. 
In the latter case we also say that $M$ realizes this type. 
Thus $M$ is countably $\Phi$-saturated
if and only if every consistent $\Phi$-type over a countable subset of $M$  is realized in $M$.

\begin{remark} 
We use the term `expression' instead of `formula' in order to avoid 
confusion with formulas
of the logic for metric structures. Also, the expressions $\|P(\bar x)\|=r$ and $\|P(\bar x)\|\leq r$ 
are identified with conditions (using the terminology of  \cite{BYBHU}) 
$\|P(\bar x)\|\in \{r\}$ and $\|P(\bar x)\|\in [0,r]$, respectively. 
Finally, instead of  $\|P(\bar x)-Q(\bar x)\|=0$ we write $P(\bar x)=Q(\bar x)$.  
\end{remark} 

Recall that the  \emph{multiplier algebra} $M(A)$ of a C*-algebra $A$
is defined to be the idealizer of $A$ in any nondegenerate representation of $A$
(see e.g., \cite{Black:Operator}).  
The \emph{corona} of $A$ is the quotient $M(A)/A$.

\begin{theorem} \label{T1} 
If $A$ is a $\sigma$-unital C*-algebra then its corona $C(A)$
is countably degree-1  saturated. 
\end{theorem}

Theorem~\ref{T1} will be proved in \S\ref{S.proof.T1}.

\begin{corollary} \label{C.Mn}
If $A$ is a $\sigma$-unital C*-algebra then $M_n(C(A))$ 
is countably degree-1 saturated for every $n\in\bbN$. 
\end{corollary}

\begin{proof} The universality property of the multiplier algebra
easily implies that $M(M_n(A))$ and $M_n(M(A))$ are isomorphic, via
the natural isomorphism that fixes $A$. Therefore 
 $M_n(C(A))$ is isomorphic to $C(M_n(A))$
and we can apply Theorem~\ref{T1}. 
\end{proof} 

The following will be proved as  Theorem~\ref{T1.0}. 

\begin{theorem} \label{T1.0.0} 
Assume $A$ is a $\sigma$-unital C*-algebra such that  
 for every separable subalgebra $B$ of $M(A)$ there  is a $B$-quasicentral 
approximate  unit for $A$ consisting of projections. 
Then its corona $C(A)$
is  countably quantifier-free saturated. 
\end{theorem}

We shall show that the Calkin algebra fails the conclusion of Theorem~\ref{T1.0.0}, and 
therefore that Theorem~\ref{T1} essentially gives an optimal conclusion in its case.

\section{Applications}
\label{S.App}
Most of our applications require only types with a single variable, or  
so-called \emph{1-types}.
We shall occasionally use shortcuts such as $a=b$ for $\|a-b\|=0$ or $a\leq b$
for $ab=a$ (the latter assuming both $a$ and $b$ are positive) in order to simplify the notation. 
We say that $c$ \emph{$\e$-realizes} type $\bt$ if 
for all conditions $\|P(x)\|\in K$ in $\bt$ we have $\|P(c)\|\in (K)_\e$. 
Therefore a type is consistent if and only if each of its finite subsets is $\e$-realized 
for each $\e>0$.

\subsection{A self-strengthening} 
 We start with a self-strengthening of the notion of approximate finite satisfiability, 
 stated only for 1-types. An obvious generalization 
 to arbitrary types is left to the reader.

\begin{lemma} \label{L.sa} If $\Phi$ includes all degree-1 *-polynomials and 
$C$ is countably $\Phi$-saturated then every countable 
$\Phi$-type $\bt$ that is approximately finitely satisfiable by  self-adjoint (positive) 
elements is realized by a self-adjoint (positive) element. 

Moreover, if $\bt$ is approximately finitely satisfiable by  self-adjoint elements whose
spectrum is included  in the interval $[r,s]$, then $\bt$ is realized 
by a self-adjoint element whose spectrum is included in $[r,s]$. 
\end{lemma} 

\begin{proof}  If $\bt$ is approximately finitely satisfiable by a self-adjoint element, 
then the type $\bt_1$ obtained by adding $x=x^*$ to $\bt$ is still approximately 
finitely satisfiable and countable, and therefore realized. Any realization of
$\bt_1$ is   a self-adjoint realization of $\bt$. 

Now assume $\bt$ is 
approximately finitely satisfiable by positive elements. 
By compactness, there is $r\in K$ such that $\bt\cup \{\|x\|=r\}$ is 
approximately finitely satisfiable by a positive element. 
Let $\bt_2=\bt\cup \{\|x\|=r, x=x^*, \|x-r\cdot 1\|\leq r\}$. 
A simple continuous functional calculus argument
shows that  for a self-adjoint $b$ we have that $b\geq 0$ if and ony if 
$\|b-\|b\|\cdot 1\|\leq \|b\|$. The proof is completed analogously to the case of a self-adjoint operator. 

Now assume $\bt$ is approximately finitely satisfiable by elements whose spectrum is 
included in $[r,s]$. Add conditions $\|x-x^*\|=0$ and  $\|x - (r+s)/2\|\leq (s-r)/2$ to $\bt$. 
The second condition is satisfied by a self-adjoint element iff its spectrum is included in the interval 
$[r,s]$. Therefore the new type is approximately finitely satisfiable and its realization 
is as required. 
\end{proof}

Note that the assumption of Lemma~\ref{L.un} is necessarily stronger than the assumption of
Lemma~\ref{L.sa} by results of \S\ref{S.Limiting}. 

\begin{lemma} \label{L.un} If 
$C$ is countably quantifier-free saturated then every countable 
quantifier-free type that is approximately finitely satisfiable by a unitary (projection)  
is realized by a unitary  (projection, respectively). 
\end{lemma} 

\begin{proof} This is just like the proof of Lemma~\ref{L.sa}, but adding conditions 
$xx^*=1$ and $x^*x=1$ in the unitary case and $x=x^*$ and $x^2=x$ in the projection case. 
\end{proof} 

In Proposition~\ref{P1}  and Proposition~\ref{P2} 
we prove that there is a countable type over the  Calkin algebra 
that is approximately finitely satisfiable by a unitary but not realized by a unitary. 
By Lemma~\ref{L.un}, Calkin algebra is not quantifier-free saturated.

\subsection{Largeness of countably saturated C*-algebras}

If $C$ is a finite-dimensional C*-algebra then its unit ball is compact, and this 
easily implies $C$ is countably saturated. 

\begin{prop} If $C$ is countably degree-1 saturated then it is either finite-dimensional
or nonseparable. 
In the latter case, $C$ even has no separable maximal abelian subalgebras. 
\end{prop}

\begin{proof} 
Assume $C$ is infinite-dimensional and let $A$ be its masa. 
Then $A$ is infinite-dimensional and there is a sequence 
of positive operators $a_n$, for $n\in \bbN$,  of norm $1$ such that $\|a_m-a_n\|=1$ 
(cf. \cite{O:findim} or \cite[Lemma~5.2]{FaHaSh:Model1}).

Assume $A$ is separable, and fix a countable dense subset $b_n$, for $n\in \bbN$, 
of its unit ball. 
 The type $\bt$ consisting of all conditions of the form  
$\|x-b_n\|\geq 1/2$ and $xb_n=b_n x$,  for $n\in \bbN$, 
together with $\|x\|=1$, 
is consistent. 
This is because each of its finite subsets is realized by $a_m$ for a large enough $m$. 
Otherwise, there are $n$, $i$ and $j$ such that $\|b_n-a_i\|<1/2$ and $\|b_n-a_j\|<1/2$. 
By countable saturation some $c\in C$ realizes $\bt$. 
Then $c\in A'\setminus A$, contradicting the assumed maximality of $A$. 
\end{proof}

\begin{lemma} \label{L.cta} 
Assume $C$ is countably $\Phi$-saturated and $\Phi$ includes all degree-1 polynomials. 
 If $A$ is a separable subalgebra of $C$ then the relative 
 commutant of $A$ is countably $\Phi$-saturated. 
 
 Moreover, if $C$ is infinite-dimensional then $A'\cap C$ is nonseparable.
 \end{lemma} 
 
 \begin{proof} 
Let $a_n$, for $n\in \bbN$, enumerate a countable dense subset of the unit ball of $A$. 
The relative commutant type over $A$, $\bt_{rc}$,  consists of all formulas of the form 
\begin{enumerate}
\item [] $\|a_n x - x a_n\|=0$, for $n\in \bbN$. 
\end{enumerate} 
If $\bt$ is a finitely approximately finitely satisfiable $\Phi$-type over $A'\cap C$ then $\bt\cup \bt_{rc}$
is a finitely approximately satisfiable $\Phi$-type over $C$. Also, an element $c$ of $C$ realizes
$\bt\cup \bt_{rc}$ if and only if $c\in A'\cap C$ and $c$ realizes $\bt$. 
Since $\bt$ was an arbitrary $\Phi$-type, countable $\Phi$-saturation of $A'\cap C$ 
follows. 

Now assume $C$ is infinite-dimensional. 
By enlarging $A$ if necessary, we can assume it is infinite-dimensional. 
 Expand $\bt_{rc}$ by 
adding all formulas of the form 
\begin{enumerate}
\item [(ii)] $\|a_n x - a_n\|\geq 1/2$. 
\end{enumerate}
We denote  the resulting type by $\bt$.
We shall prove that $\bt$ is finitely approximately realizable. 
This follows from the proof of 
 \cite[Lemma~5.2]{FaHaSh:Model1} 
and we refer the reader to this paper for details. 
First, if $A$ is a continuous trace, infinite-dimensional 
 algebra then its center $\cZ(A)$ is infinite-dimensional. 
 Therefore $\cZ(A)$ includes a sequence of contractions $f_n$, for $n\in \bbN$, 
 such that $\|f_m - f_n\|=1$ if $m\neq n$ (this is a consequence of Gelfand--Naimark 
 theorem, see e.g., the proof of  \cite[Lemma~5.4]{FaHaSh:Model1}), and therefore 
  $\bt$ is approximately finitely satisfiable by $f_m$'s. 

If $A$ is not a continuous trace algebra, then by \cite[Theorem 2.4]{AkPed}
it has a nontrivial central sequence. Elements of such a sequence 
witness that $\bt$ is finitely approximately realizable. 

By countable saturation, $\bt$ is realized in $C$. 
A realization of $\bt$ in $C$ is at a distance $\geq 1/2$ from $A$, and therefore we have proved
that $A'\cap C\not\subseteq A$. 

Now assume $A$ is a separable, not necessarily infinite-dimensional, subalgebra of 
$C$. Since $C$ is infinite-dimensional, find infinite-dimensional $A_0$ such that 
$A\subseteq A_0\subseteq C$. By using the above, build an increasing chain 
of separable subalgebras of $C$, $A_\gamma$, for $\gamma<\aleph_1$, such that
$A_\gamma'\cap A_{\gamma+1}$ is nontrivial for all $\gamma$. This shows that $A'\cap C$
intersects $A_{\gamma+1}\setminus A_\gamma$ for all $\gamma$, and it is therefore 
nonseparable. 
 \end{proof}

\subsection{Properties of countably degree-1 saturated C*-algebras} 
In the following there is a clear analogy with the theory of gaps in $\cP(\bbN)/\Fin$. 
\begin{definition}\label{D.gap}
 Two subalgebras $A,B$ of an algebra $C$  are \emph{orthogonal} if $ab=0$ for all $a\in A$ 
and $b\in B$. They are \emph{separated} if there is a positive element $c\in C$ such that 
$cac=a$ for all $a\in A$ and $cb=0$ for all $b\in B$. 
\end{definition}

A C*-algebra $C$ has AA-CRISP (asymptotically abelian, countable Riesz separation property) 
if the following holds: Assume $a_n, b_n$, for $n\in \bbN$, 
are positive elements of $C$ such that 
\[
a_n\leq a_{n+1}\leq b_{n+1}\leq b_n
\]
for all $n$. Furthermore assume $D$ is a separable subset of $C$ such that 
for every $d\in D$ we have 
\[
\lim_n \|[a_n,d]\|=0.
\] 
Then there exists a positive $c\in C$ such that $a_n\leq c\leq b_n$ for all $n$
and $[c,d]=0$ for all $d\in D$. 

By Theorem~\ref{T1} the following is a strengthening of the result that 
 every corona of a $\sigma$-unital C*-algebra has AA-CRISP  
(\cite[Corollary~6.7]{Pede:Corona}).

\begin{prop}\label{P.AA} Every countably degree-1 saturated C*-algebra $C$ has AA-CRISP. 
\end{prop}

\begin{proof} By scaling, we may assume that  $\|b_1\|=1$. 
Fix a countable dense subset $\{d_n\}$ of $D$ and 
let $\bt$ be the type consisting of the following conditions: 
$a_n\leq x$, $x\leq b_n$ and $[d_n,x]=0$, for all $n\in \bbN$. 
If $\bt_0$ is any finite subset of $\bt$ and $\e>0$,
 then for a large enough $n$ 
we have that  $a_n$ $\e$-approximately realizes $\bt_0$. 
By countable saturation of $C$, some $c\in C$ realizes $\bt$. 
This $c$ satisfies the requirements of the AA-CRISP for $a_n, b_n$ and $D$. 
\end{proof} 

Recall that a C*-algebra
 $C$ is an SAW*-algebra if any two  $\sigma$-unital subalgebras $A$ and $B$ 
of $C$ are orthogonal if and only if they are separated. 
By Theorem~\ref{T1} the following is a strengthening of the result that 
 every corona of a $\sigma$-unital C*-algebra is an SAW*-algebra 
(\cite[Corollary~7.5]{Pede:Corona}). (By \cite[Theorem~7.4]{Pede:Corona}, CRISP 
implies SAW* but we include a simple direct proof below.) 

\begin{prop} \label{P.SAW} 
Every countably degree-1 saturated C*-algebra $C$ is an SAW*-algebra. 
\end{prop}

\begin{proof} Assume $A$ and $B$ are $\sigma$-unital subalgebras of $C$
such that $ab=0$ for all $a\in A$ and all $b\in B$. 
Let $a_n$, for $n\in \bbN$ and $b_n$, for $n\in \bbN$, 
be an approximate identity of $A$ and $B$, 
respectively. 
 Consider type $\bt_{AB}$ consisting of the following expressions, for all $n$. 
\begin{enumerate}
\item [(i)] $a_n x=a_n$, 
\item [(ii)] $xb_n=0$
\item [(iii)] $x=x^*$. 
\end{enumerate}
Every finite subset of $\bt_{AB}$ is $\e$-realized by $a_n$ for a large enough $n$. 
If $c$ realizes $\bt_{AB}$, then $ac=a$ for all $a\in A$ and $cb=0$ for all $b\in B$. 
Moreover, $c$ is self-adjoint by (iii) and $|c|$ still satisfies the above. 
\end{proof}

Assume $B$,  $C$ and $D$ are subalgebras of a C*-algebra $M$. 
 We say that $D$ \emph{derives} 
$B$ if for every $d\in D$ the derivation $\delta_d(x)=dx-xd$ maps $B$ into itself. 
The following is an extension of 
 Higson's formulation of Kasparov's Technical Theorem (\cite{Hig:Technical}, 
 also  \cite[Theorem~8.1]{Pede:Corona}).

We say that a C*-algebra $M$ has  KTT if the following holds: 
Assume $A,B$, and $C$ are subalgebras of $M$ such that $A\perp B$
and $C$ derives $B$.  Furthermore assume $A$ and $B$ are $\sigma$-unital and 
$C$ is separable. Then there is a positive element $d\in M$ such that 
$d\in C'\cap M$, the map $x\mapsto xd$ is the identity on $B$, and the 
map $x\mapsto dx$ annihilates $A$.

\begin{prop} \label{P.KTT} Every   countably degree-1 saturated C*-algebra  
has KTT. 
\end{prop}

\begin{proof} Assume $A,B$ and $C$ are as above. 
Since $B$ is $\sigma$-unital we can 
fix a strictly positive element $b\in B$. Then $b^{1/n}$, for $n\in \bbN$, is an approximate
unit for $B$. An easy computation demonstrates 
that for every $c\in C$ the commutators $[b^{1/n}, c]$ strictly converge to $0$ (see 
the first paragraph of the proof of Theorem 8.1 in 
 \cite{Pede:Corona}). They therefore converge to 0 weakly. 
 The Hahn--Banach theorem combined with the separability of $C$ 
 now shows that one can extract an approximate unit $(e_m)$ 
 for $B$ in the convex closure  of $\{b^{1/n} : n\in \bbN\}$ such that 
 the commutators $[e_m,c]$ norm-converge to 0 for every $c\in C$.

  In other words,   $B$ has an approximate unit $(e_m)$
which is 
$C$-quasicentral. 
 Fix a countable approximate unit $(f_n)$ of $A$ and a countable dense subset $\{c_m\}$ of $C$.  
Consider the type $\bt$ consisting of the following conditions, for all $m$ and all $n$. 
\begin{align*}
\| e_n x-e_n\|&=0\\
\|xf_n\|&=0\\
\| [c_m, x]\|&=0\\ 
\|x-x^*\|&=0. 
\end{align*}
For every finite subset $F$ of this type and every $\e>0$ there exists an $m$ large enough so 
that all the conditions in $F$ are $\e$-satisfied with $x=e_m$. Therefore the type $\bt$ is consistent
and by countable degree-1 saturation it is satisfied by some $d_0$. Then  $d=|d_0|$ is as required. 
\end{proof}

A C*-algebra $M$ is \emph{sub-Stonean} if for all $b$ and $c$ in $M$ such that 
$bc=0$ there are positive contractions $f$ and $g$ such that $bf=b$, $gc=c$ and $fg=0$. 
By considering $B=C^*(b)$ and $C=C^*(c)$ and noting that $B$ and $C$ are
orthogonal, one easily sees that every SAW* algebra is sub-Stonean. 
The following strengthening  was introduced by Kirchberg  \cite{Kirc:Central}. 

\begin{definition} \label{D.ssS} 
A C*-algebra $C$ is \emph{$\sigma$-sub-Stonean} if for every separable subalgebra $A$ of $C$ and all positive $b$ and $c$ in $C$
such that $bAc=\{0\}$ there are contractions $f$ and $g$ in $A'\cap C$
such that $fg=0$, $fb=b$ and $gc=c$. 
\end{definition} 

The fact that for a separable C*-algebra $A$ the relative commutant of $A$ in 
its ultrapower associated with a nonprincipal ultrafilter on $\bbN$ (as well as the 
related algebra $F(A)=(A'\cap A^{\cU})/\Ann(A,A^{\cU})$, see \cite{Kirc:Central}) 
is  $\sigma$-sub Stonean was used in \cite{Kirc:Central} to deduce many other properties
of the relative commutant. Several proofs in  \cite{Kirc:Central}, in particular the ones in the appendix, 
can easily be recast in the language of logic for metric structures. 

Before we strengthen Kirchberg's result by proving 
countably degree-1 saturated algebras  are  $\sigma$-sub-Stonean
(Proposition~\ref{P.ssS}) we prove a lemma. 

\begin{lemma}\label{L.ideal} Assume $M$ is countably degree-1 saturated and 
$B$ is a separable subalgebra. If $I$ is a (closed, two-sided)  ideal of $B$ 
then there is a contraction $f\in M\cap B'$ such that $af=a$ for all $a\in I$. 

If moreover $c\in M$ is such that $Ic=\{0\}$, then we can choose $f$ so that 
$fc=0$ and $fIc=\{0\}$. 
\end{lemma} 

\begin{proof} 
Fix a countable dense subset $a_n$, for $n\in \bbN$, of $I$ and
a countable dense subset $b_n$, for $n\in \bbN$, on $B$. Consider type $\bt$
consisting of the following conditions. 
\begin{enumerate}
\item $\|a_n x-a_n\|=0$ for all $n\in \bbN$, 
\item $\|b_n x-x b_n\|=0$ for all $n\in \bbN$. 
\item $xc=0$, and
\item $xa_n c=0$ for all $n\in \bbN$.  
\end{enumerate}
We prove that $\bt$ is consistent, and moreover that it is finitely approximately satisfiable
by a contraction. By \cite{Arv:Notes} $I$ has a $B$-quasicentral  approximate unit 
$e_n$, for $n\in \bbN$, consisting of positive elements. Since $Bc=\{0\}$ 
we have $e_n c=0$, as well as $e_n a_m c=0$  for all $m$ and all 
$n$. 
Therefore every finite fragment of 
$\bt$ is arbitrarily well approximately satisfiable by  $e_n$ for all large enough $n$. 
By Lemma~\ref{L.sa} (applied with $[r,s]=[0,1]$) and saturation of $M$ 
there is a contraction $f\in M$ that realizes $\bt$.  
Then $fa=a$ for all $a\in I$,  $f\in B'\cap M$, $fAc=\{0\}$,  and $fc=0$, as required. 
\end{proof} 

\begin{prop} \label{P.ssS} 
Every countably degree-1 saturated C*-algebra is $\sigma$-sub-Stonean. 
\end{prop}

\begin{proof} Fix $A,b$ and $c$ as in Definition~\ref{D.ssS}. 
By applying Lemma~\ref{L.ideal} find a contraction $f\in M\cap A'$ such that 
 $bf=b$, $fc=0$ and $fAc=\{0\}$. 
Now let $C=C^*(A,c)$ and let $J$ be the ideal of $C$ generated by $c$. 
By applying Lemma~\ref{L.ideal} again (with left and right sides switched) 
with $c$ replaced by $f$ we find a contraction $g\in M\cap A'$ such that 
$fg=0$, and $gc=c$. 
\end{proof}

By Theorem~\ref{T1} the following is a strengthening of the result that 
 every derivation of a separable subalgebra of the corona of a 
   $\sigma$-unital C*-algebra is inner (\cite[Theorem~10.1]{Pede:Corona}). 

\begin{prop} \label{P.derivation} 
Assume $C$ is a 
countably degree-1 saturated C*-algebra and $B$ is a separable subalgebra. 
Then every derivation $\delta$ of $B$ is of the form $\delta_c$ for some $c\in C$. 
\end{prop}

\begin{proof} Fix a countable dense subset $B_0$ of $B$. 
Consider type $\bt_{\delta}$ consisting of following conditions, 
for $b\in B_0$. 
\begin{enumerate}
\item [(i)]  $\| xb - bx-\delta(b)\|=0$. 
\end{enumerate}
By \cite[8.6.12]{Pede:C*}
this type is consistent and if $c$ realizes it then $\delta(b)=\delta_c(b)$ for all $b\in B$. 
\end{proof}

\subsection{Automorphisms}

In \cite{PhWe:Calkin} the authors proved that the Continuum Hypothesis implies that the
Calkin algebra has $2^{\aleph_1}$ outer automorphisms. Since $\kappa<2^\kappa$ for all cardinals
$\kappa$, this conclusion implies that Calkin algebra has outer automorphisms.  
A simpler proof of Phillips--Weaver result was given in 
\cite{Fa:All}. The proof of Theorem~\ref{T.auto}  
below is in the spirit of \cite{PhWe:Calkin}, but instead of 
results about KK-theory it uses countable quantifier-free saturation.

Recall that the \emph{character density} of a C*-algebra is the smallest cardinality of a dense subset. 
The following remark refers to the full countable saturation in logic for countable
structures, not considered in the 
present paper (cf. \cite{FaHaSh:Model2}). 
The standard back-and-forth method shows that a fully countably saturated 
C*-algebra of character density $\aleph_1$  has $2^{\aleph_1}$ automorphisms. 
Therefore, the Continuum Hypothesis implies that $M$ has $2^{\aleph_1}$ automorphisms
whenever $M$ is an ultrapower of a separable C*-algebra, a relative commutant of 
a separable C*-algebra in its ultrapower, or an algebra of the form $\prod_n A_n/\bigoplus_n A_n$
for a sequence of separable unital C*-algebras $A_n$, for $n\in \bbN$.
Since $\aleph_1$ is always less than $2^{\aleph_1}$, in this situation, the automorphism group is strictly larger than 
the group of inner automorphisms. 
These issues will be treated in an upcoming paper joint with David Sherman. In the following we show 
how to construct $2^{\aleph_1}$ automorphisms in a situation where algebra is only quantifier-free saturated.

\begin{theorem} \label{T.auto} 
If $C$ is a countably quantifier-free 
 saturated C*-algebra of 
 character density $\aleph_1$ whose center is separable
then $C$ has $2^{\aleph_1}$  automorphisms. 
\end{theorem}

Before proceeding to prove Theorem~\ref{T.auto} we note that 
every countably saturated metric structure of character density $\aleph_1$
has $2^{\aleph_1}$ automorphisms. We don't know whether 
the Continuum Hypothesis implies that every corona of a separable C*-algebra
has $2^{\aleph_1}$ automorphisms (but see \cite{CoFa:Automorphisms}).

By Theorem~\ref{T.auto} and Theorem~\ref{T1.0} we have the following:

\begin{corollary} Assume the Continuum Hypothesis. 
Assume $A$ is a C*-algebra such that 
for every separable subalgebra $B$ of $M(A)$ there  is a $B$-quasicentral 
approximate  unit for $A$ consisting of projections and the center of $C(A)$ 
is separable.  
Then $C(A)$ has $2^{\aleph_1}$ outer automorphisms. \qed 
\end{corollary}

Recall  that an  automorphism $\Phi$ of a C*-algebra $C$
 is \emph{approximately inner} if for every
$\e > 0$ and every finite set $F$, there is a unitary $u$ such that
$\| \Phi (a) - u a u^* \| < \e$ for all $a\in F$.
An approximately inner *-isomorphism from a subalgebra of $C$ into $C$ is defined analogously.

The conclusion of the following lemma fails for the Calkin algebra (cf. Proposition~\ref{P2}). 

\begin{lemma}\label{L.auto.1} Assume $C$ is a countably quantifier-free saturated C*-algebra
and $B$ is its separable subalgebra. If $\Phi\colon B\to C$ is an approximately inner *-isomorphism
 then there is a unitary $u\in C$ such that $\Phi(b)=ubu^*$ for all $b\in B$. 
\end{lemma}

\begin{proof} This is essentially a consequence of Lemma~\ref{L.un}. 
Fix a countable dense subset $B_0$ of $B$. 
Consider the type $\bt_{\Phi}$ consisting of all conditions of the form 
$\| xbx^* - \Phi(b)\|=0$
for $b\in B_0$ together with $xx^*=1$ and $x^*x=1$. 
The assumption that $\Phi$ is  approximately inner is equivalent to 
the assertion that $\bt_{\Phi}$ is consistent. Since $B_0$ is countable, 
by countable quantifier-free saturation 
there exists $u\in C(A)$ that realizes $\bt_{\Phi}$.  Such $u$ is a unitary which implements $\Phi$. 
\end{proof}

\begin{lemma} \label{L.cta.1} 
Assume $C$ is countably quantifier-free saturated, simple C*-algebra 
whose center is separable. If   $\Phi$ is an automorphism of 
$C$ and $A$ is a separable subalgebra of $C$ then there is an automorphism $\Phi'$ of $C$
distinct from $\Phi$ whose restriction to $A$ is identical to the restriction of $\Phi$ to $A$. 
Moreover, if $\Phi$ is inner then $\Phi'$ can be chosen to be inner. 
\end{lemma} 

\begin{proof}  
By Lemma~\ref{L.cta}, we can find a non-central unitary $u\in A'\cap C$. 
Therefore $\Phi'= \Phi\circ \Ad u$ is as required. 
\end{proof}

\begin{proof}[Proof of Theorem~\ref{T.auto}] 
By using Lemma~\ref{L.auto.1} and Lemma~\ref{L.cta.1}
we can construct a complete binary tree of height $\aleph_1$ whose branches correspond 
to distinct automorphisms. This standard construction is similar to the one given in \cite{PhWe:Calkin}
but much easier, since in our case the limit stages are covered by Lemma~\ref{L.auto.1}, and
in \cite{PhWe:Calkin} most of the effort was made in the limit stages.   
\end{proof} 

\section{Proofs} 
\label{S.Proofs}

Recall that $A$ is an \emph{essential ideal} of $C$ if
no nonzero element of $C$ annihilates $A$. 
The \emph{strict topology} on $M(A)$ is the topology induced by the family of 
seminorms $\|(x-y)a\|$, where $a$ ranges over $A$. 
If $A$ is separable then the strict topology on $M(A)$ has a compatible 
metric, $\|(x-y)a\|$, where $a$ is any strictly positive element of $A$. 

We note that for any sequence of C*-algebras $A_n$, for $n\in \bbN$, 
the algebra $\prod_n A_n/\bigoplus_n A_n$ is fully countably saturated. 
This is a straightforward analogue of a well-known result in classical model theory (cf. \cite{FaHaSh:Model2}, \cite{BYBHU}). 

\subsection{Quantifier-free saturation}
\label{S.proof.T1}
The proof of Theorem~\ref{T1.0.0} 
 is a warmup for the proof of Theorem~\ref{T1} given in the next subsection. 
In Proposition~\ref{P2} we shall see that the conclusion of Theorem~\ref{T1.0} does not follow from the assumptions of 
Theorem~\ref{T1}. 
Let us start by recalling the statement of Theorem~\ref{T1.0.0}. 

\begin{theorem} \label{T1.0} 
Assume $A$ is a $\sigma$-unital C*-algebra such that  
 for every separable subalgebra $B$ of $M(A)$ there  is a $B$-quasicentral 
approximate  unit for $A$ consisting of projections. 
Then its corona $C(A)$
is  countably quantifier-free saturated. 
\end{theorem}

In this section and elsewhere we shall write $\bar b$ for an $n$-tuple, 
hence
\[
\bar b=(b_1, \dots, b_n)
\]
 (with $n$ clear from the context) 
in order to simplify the notation. We also write 
\[
q\bar b =(q b_1, \dots, q b_n). 
\]
In our proof of Theorem~\ref{T1.0} we shall need the following  fact. 

\begin{lemma} \label{L.P.K} Assume $P(x_1,\dots, x_n)$ 
is a *-polynomial with coefficients in a C*-algebra $C$. 
Then there is constant $K<\infty$, depending only on $P$, such that 
for all $a$ and $b_1,\dots, b_n$ in $C$  we have
\[
\| [ a, P(\bar b)]\|\leq K \max_c  \|[a,c]\| \|a\| \max_{j\leq n}   \|b_j\|
\]
where $c$ ranges over coefficients of $P$ and $b_1, \dots, b_n$. 

If in addition $q$ is a projection then we have 
\[
\| q P(\bar b)- q P( q \bar b) q \|\leq K \max_c  \|[q,c]\| \|a\| \max_{j\leq n}  \|b_j\|. 
\]
\end{lemma} 

\begin{proof} The existence of 
constant $K$ satisfying the first inequality can be 
 proved  by a straightforward induction on the complexity of $P$. 
 For the second inequality use the first one and the fact that $q=q^{d+1}$, where $d$ is the degree of $P$
 in order to find a large enough $K$. 
 \end{proof}

\begin{proof}[Proof of Theorem~\ref{T1.0}]  Fix a countable quantifier-free type $\bt$ over $C(A)$
and enumerate all polynomials occurring in it as $P_n(\bar x)$, for $n\in \bbN$. 
By re-enumerating and adding redundancies we may assume that all variables of $P_n$ are among   
$x_1,\dots, x_n$. Let $P_n^0(\bar x)$ be a polynomial over $M(A)$ 
corresponding to  $P_n(\bar x)$. Let $K_n$ be a constant corresponding to $P_n^0$ as 
given by Lemma~\ref{L.P.K}.   
Let $B$ be a separable subalgebra of $M(A)$ such that all coefficients of all polynomials
$P_n^0(\bar x)$  belong to  $B$. 

Let $r_n$ for $n\in \bbN$ be  such that $\bt$ is the set of conditions
$\|P_n(\bar x)\|=r_n$ for $n\in \bbN$. 
For all $n$ fix $b^n_1,\dots, b^n_n$ such that 
\[
|\|\pi(P_j^0(b_1^n,\dots, b^n_n))\|-r_n| <2^{-n}
\]
for all $j\leq n$ and $\|b^n_k\|\leq 2$. The latter is possible by our assumption that the condition 
$\|x_n\|\leq 1$ belongs to $\bt$ for all $k$. 

Let $q_n$, for $n\in \bbN$, be a $B$-quasicentral  approximate unit for $A$ consisting of 
projections. 
By going to a subsequence we may assume the following apply for all $j\leq n$  (with $q_0=0$): 
\begin{enumerate}
\popcounter
\item  \label{I.T1.0.1} $\|[q_n,a]\|<2^{-n}K_n^{-1}$ 
 when $a$ ranges over coefficients of  $P_j^0$ and all $b^j_1,\dots, b^j_j$,
 \item $|\|(q_{n+1}-q_n) P_j^0(b^j_1,\dots, b^j_j) (q_{n+1}-q_n)\|-r_n|<1/n$, 
 \pushcounter
 \end{enumerate}
 Let 
 \[
 p_n = q_{n+1}-q_n
 \]
 For every $k$ the series $\sum_n p_n b^n_k p_n$ is convergent with respect to the strict topology. 
 Let $b_k$ be equal to the sum of this series. By the second inequality of 
 Lemma~\ref{L.P.K} and \eqref{I.T1.0.1}
 we have that for all $k\leq n$
\begin{enumerate}
\popcounter
\item $ |\|p_n P^0_k(b_1, \dots, b_k) \|-\|p_n P^0(p_n b_1 p_n, \dots, p_n b_k p_n ) p_n\| |<2^{-n}$.
\pushcounter
\end{enumerate}
Since $p_n b_k p_n =p_n b^n_k p_n$, we conclude that 
\[
\|P_k(\pi(\bar b))\|=\|\pi(P^0_j(\bar b))\|=\limsup_n \| p_n P^0_j(p_n b^n_1 p_n,\dots , p_n b^n_j p_n)\| 
=r_n. 
\]
Therefore $\pi(b_n)$, for $n\in \bbN$, realizes $\bt $ in $C(A)$. 
 \end{proof} 

\subsection{Degree-1 saturation}
We shall use   \cite[Corollary~6.3]{Pede:Corona} which states that 
if $0\leq a\leq 1$ and $\|b\|=1$, then $\|[a,b]\|\leq \e\leq 1/4$ 
implies $\|[a^{1/2}, b]\|\leq 5\e^{1/2}/4$. We shall also need the following 
lemma. 

\begin{lemma} \label{L.triv}
Assume $a$ and $b$ are positive operators. 
Then $\|a+b\|\geq \max(\|a\|, \|b\|)$. 
\end{lemma} 
\begin{proof} We may assume $1=\|a\|\geq \|b\|$. Fix $\e>0$ and 
let $\xi$ be a unit vector such that $\eta=\xi-a\xi$ satisfies $\|\eta\|< \e$. Then 
$\Re(a\xi|b\xi)=\Re(\xi|b\xi)+\Re(\eta|b\xi)\geq \Re(\eta |b\xi)>  -\e$ since $b\geq 0$. 
We therefore have 
 \begin{align*}
\|(a+b)\xi\|^2&=((a+b)\xi| (a+b)\xi)\\
&=\|a\xi\|^2+\|b\xi\|^2+2 \Re(a\xi|b\xi)> 1+\|b\xi\|^2 -2\e
\end{align*} 
and since $\e>0$ was arbitrary the conclusion follows. 
\end{proof}

\begin{lemma} \label{L1}
Assume $M$ is a  C*-algebra and a $\sigma$-unital C*-algebra
$A$ is an essential ideal of $M$. Furthermore assume
$F_n$, for $n\in \bbN$, 
is an increasing sequence of finite subsets of the unit ball of $M$ and $\e_n$, for $n\in \bbN$, 
is a decreasing sequence of positive numbers converging to 0. 
Then $A$ has an approximate unit $e_n$, for $n\in \bbN$ such that  with  (setting $e_{-1}=0$)
\[
f_n =(e_{n+1}-e_n)^{1/2}
\]
for all $n$ and all $a\in F_n$ we have the following:
\begin{enumerate}
\item \label{T1.1} $\|[a,f_n]\|\leq \e_n$,
\item \label{T1.2} $\| f_n a f_n \|\geq \|\pi(a)\|-\e_n$ (where $\pi\colon M\to M/A$ is the quotient map), 
\item\label{T1.3}  $\|f_m f_n\|=0 $ if $|m-n|\geq 2$,  
\item\label{T1.4}  $\|[f_n,f_{n+1}]\|\leq \e_n$. 
\pushcounter
\end{enumerate}
\end{lemma}
\begin{proof}  
Let $\delta_n=(4\e_n/25)^2$. 
 By \cite[\S 1]{Arv:Notes} inside the convex closure of any approximate unit of $A$ 
 we can find an approximate unit $(e_n^0)$ of 
$A$ such that 
\begin{enumerate}
\popcounter
\item $\|e_n^0 a - a e_n^0\|\leq \delta_n$  for all $a\in F_n\cup \{e^0_i: i<n\}$. 
\pushcounter
\end{enumerate}
In order to take care of the condition \eqref{T1.3} we do the following.  
Let $h$ be a strictly positive element of $A$. By continuous functional calculus we choose 
an approximate unit $(e_n^{-1})$ of $A$ satisfying \eqref{T1.3}.
Applying Arveson's construction  to subsequences of $(e_n^{-1})$,  $n\in \bbN$, preserves \eqref{T1.3}.    
By going to a subsequence, we can assure that $\|e_n^0 e_{n+1}^0 - e_n^0\|\leq \delta_n$ 
for all $n$. 

Then for every subsequence $(e_n)$ of $(e_n^0)$ and $f_n$ defined as above we will 
have \eqref{T1.1} and \eqref{T1.4} by the choice of
 $\delta_n$ and  \cite[Corollary~6.3]{Pede:Corona}.
 Also, again using this corollary if $n\geq m+2$ then  
 \[
\| f_m f_n\|^2=\|f_n f_m^2 f_n\|\leq \|(e^0_{m+1}-e^0_n)(e^0_{n+1}-e^0_n)\|+\frac{\e_n}5
\leq \e_n. 
\]
Since $A$ is an essential ideal of $M$, there is a faithful representation $\alpha\colon M\to B(H)$
such that $\alpha[A]$ is an essential ideal of $B(H)$ (this is essentially by 
\cite[II.6.1.6]{Black:Operator}). In particular $\alpha(e_n)$ strongly converges to $1_H$. 
 Therefore for every $a\in M$, $m\in \bbN$,  and $\e>0$ 
  there is $n$ large enough so that $\| \alpha(a) (e_n-e_m)\|\geq \|\alpha(a)\|-\e$. 
  Using this observation we can recursively find a subsequence $(e_n)$ of $(e_n^0)$ 
  such that  $\|(e_{n+1}-e_n) a\|\geq \|\pi(a)\|-\delta_n$ for all $a\in F_n$. 
  Therefore $\|f_n a f_n\|\geq \|\pi(a)\| -\e_n$ for all $a\in F_n$ and  
  \eqref{T1.2} holds. 
    \end{proof}

Fix a $\sigma$-unital C*-algebra $A$; let $M=M(A)$, and   $\e_n=2^{-n}$. Now by applying Lemma~\ref{L1}  
we get  $A,M, F_n, (e_n)$ and $(f_n)$, 
 for $n\in \bbN$. We shall show that in this situation these objects have the additional 
 properties in 
 formulas \eqref{I.f2}--\eqref{I.inX} below. 
\begin{enumerate}
\popcounter
\item \label{I.f2}
The series $\sum_n f_n^2$ strictly converges to 1. 
\pushcounter
\end{enumerate}
Since $A$ is $\sigma$-unital, we can pick a strictly positive $a\in A$. 
Therefore the strict topology is given by compatible  metric $d(b,c)=\| a(b-c)\|$. 
Fix $\e>0$. 
Let $n$ be large enough so that $\| a e_{n+1}-a\|<\e$. Since $1-e_{n+1}=\sum_{j=n+1}^\infty f_j^2$, 
\eqref{I.f2} follows. 
\begin{enumerate}
\popcounter
\item \label{I.bj}
For every sequence $(b_j)$ in the unit ball of $M$  
the series $\sum_j f_j b_j f_j$ is strictly convergent. 
\pushcounter
\end{enumerate}
We first note that $0\leq c\leq d$ implies $\|cb\|\leq \| db\|$ for all $b$. 
This is because $\|cb\|^2=\|b^*c^2b\|\leq \|b^* d^2 b\|=\|db\|^2$.

Since every element $b$ of a C*-algebra 
is a linear combination of four positive elements $b=c_0-c_1+ic_2 -ic_3$, we may assume 
$b_j\geq 0$ for all $j$. 
Fix $\e>0$ and find $n$ large enough so that (with $a\in A$ strictly positive) 
$\|\sum_{j=n}^\infty (f_j^2)a\|<\e$. 
Then $0\leq \sum_{j\geq n} f_j b_j f_j\leq \sum_{j\geq n} f_j^2$. 
Therefore by the above inequality applied 
with $c=\sum_{j\geq n} f_j b_j f_j$ and $d=\sum_{j\geq n} f_j^2$
we have $\|ca\|\leq \|da\|\leq \e$. 

\begin{enumerate}
\popcounter
\item \label{I.sum1} 
$\|\sum_j f_j x_j f_j \|\leq \sup_j \| f_j x_j f_j\|$ for every norm-bounded sequence $(x_j)$. 
\item \label{I.sum1+}  If in addition $\sup_j \|f_j x_j f_j\|=\sup_j \|x_j\|$ then 
we moreover have the equality in \eqref{I.sum1}. 
\pushcounter
\end{enumerate}
%
In order to prove \eqref{I.sum1} consider the C*-algebra $N=\prod_{\bbN} M$. 
Each map 
\[
N\ni (x_k)_{k\in \bbN}\mapsto f_j x_j f_j\in M
\]
 for $j\in \bbN$
is completely positive on $N$,  and therefore for each $n\in \bbN$  the 
map $(x_k)_{k\in \bbN}\mapsto \sum_{j\leq n} f_j x_j f_j$ is completely positive as well. 
The  supremum of these maps is
also  a completely positive map. By the assumption that  $\sum_ j f_j^2=1$ 
this map is also unital, and therefore of norm $1$. The inequality \eqref{I.sum1} follows. 

In order to prove \eqref{I.sum1+} let $\alpha=\sup_j \|x_j\|$. We may assume $\alpha=1$. 
Fix $\e>0$, unit vector $\xi$, 
and $n$ such that $\| (f_n x_n f_n) \xi\|>1-\e$. 
Then $\|f_n \xi\|\geq 1-\e$ and therefore $|(f_n^2\xi|\xi)|=\|f_n\xi\|\geq 1-\e$ 
and this implies that $\|\xi-f_n^2\xi\|\leq \e$.  
Since $\sum_j f_j^2=1$, this shows that  
$\|\sum_j (f_j x_j f_j) \xi\|\approx \|(f_n x_n f_n) \xi\|$ and the conclusion follows. 

Recall that $\pi\colon M(A)\to C(A)$ is the quotient map. In the following the norm on the left-hand side of the equality is computed in the corona and the norm on the right-hand side is computed in the multiplier algebra.

\begin{enumerate}
\popcounter
\item \label{I.sum2} 
$\|\pi(\sum_j f_j x_j f_j )\|=\limsup_j \| f_j x_j f_j\|$ for every bounded sequence $(x_j)$ 
such that $\sup_j \| f_j x_j f_j\|=\sup_j \|x_j\|$. 
\pushcounter
\end{enumerate}
Since $\sum_{j=0}^\infty f_j x_j f_j- \sum_{j=m}^\infty f_j x_j f_j$ is in $A$ for all $m\in \bbN$, 
the inequality $\leq$  follows from \eqref{I.sum1} and $\|\pi(a)\|\leq \|a\|$. 
Similarly, $\geq$ follows from \eqref{I.sum1+}. 

The converse inequality follows by Lemma~\ref{L.triv}. 
\begin{enumerate}
\popcounter
\item \label{I.Xfn} 
$X_{(f_n)}=\{a\in M: \sum_n \|[a,f_n]\|<\infty\}$
 is a subalgebra of $M$ including $C^*(\bigcup_n F_n)$.  
\pushcounter
\end{enumerate}
Since $b\in F_j$ implies $\|[b,f_n]\|\leq 2^{-n}$ for all $n\geq j$, we have  $\bigcup_j F_j \subseteq X_{(f_n)}$. 

For $a$ and $b$ in $M$ we have 
$[a+b,f_n]=[a,f_n]+[b,f_n]$, $\|[a^*,f_n]\|=\|[a,f_n]\|$
and $\|[ab,f_n]\|\leq \|a\|\cdot \| [b,f_n]\|+\|b\|\cdot \|[a,f_n]\|$. 
Therefore $X_{(f_n)}$ is a *-subalgebra of $M$. 

$X_{(f_n)}$ is not necessarily norm-closed but this will be of no consequence.

\begin{enumerate}
\popcounter
\item \label{I.binA} 
The map 
$\Lambda=\Lambda_{(f_n)}$ from  $M$ into $M$ 
defined by 
\[
\textstyle\Lambda (a)=\sum_n f_n a f_n
\]
is completely positive  
 and it satisfies $b-\Lambda(b)\in A$ for all $b\in X_{(f_n)}$. 
\pushcounter
\end{enumerate}
Note that $\|\Lambda(b)\|\leq \|b\|$ by \eqref{I.sum1}, and the map is clearly completely 
positive. Fix $b\in X_{(f_n)}$ and $\e>0$. 
Since $b\in X_{(f_n)}$ the series  $\delta_j=\|f_j b-b f_j\|$ is convergent, and we can 
pick  $n$ large enough to have $\sum_{j\geq n} \| f_j b-b f_j\|\leq e$. 
We write 
$c\sim_A d$ for $c-d\in A$ and $c\sim_{\e} d$ for $\|c-d\|\leq \e$ (clearly the latter is not an equivalence relation). 
We have $b-(1-e_n)b\in A$ and $\sum_{j\leq n} f_j b f_j\in A$. Also, 
with $\delta=\sum_{j\geq n}\delta_j$ we have 
\[
\textstyle (1-e_n)b =\sum_{j=n}^\infty f_j^2 b\sim_\delta\sum_{j=n}^\infty f_j b f_j
\]
and the conclusion follows. 

\begin{enumerate}
\popcounter
\item \label{I.inX}
If $\sup_j \|x_j\|<\infty$  
and $\delta_j = \sup_{i\geq j} \|[x_j,f_i]\|$ are such that $\sum_j \delta_j<\infty$, 
then $x=\sum_j f_j x_j f_j$ belongs to $ X_{(f_n)}$. 
\pushcounter
\end{enumerate}
We have $f_n (\sum_j f_j x_j f_j)\sim_{4\e_{n-1}} f_n (\sum_{j=n-1}^{n+1} f_j x_j f_j)$. 
Since $\|[f_k,f_{k+1}]\|\leq \e_k$ 
we have
\[
\textstyle\|[x,f_n]\|\leq \sum_{j=n-1}^{n+1} \|[f_j x_j f_j, f_n]\|
\leq 3(2\e_{n-1}+\delta_{n-1})
\]
and the conclusion follows.

\begin{proof}[Proof of Theorem~\ref{T1}]
Fix a $\sigma$-unital algebra $A$ and let $\pi\colon M(A)\to C(A)$ be the quotient map.

Fix degree-1 *-polynomials $P_n(\bar x)$ with coefficients in $C(A)$
and compact subsets $K_n\subseteq \bbR$ such that for every $n$ the system
\begin{enumerate}
\popcounter
\item \label{I.approx}
$\|P_j(\bar x)\|\in (K_j)_{1/n}\qquad \text{ for all } j\leq n$
\pushcounter
\end{enumerate}
has a solution in $C(A)$. 
Without a loss of generality all the inequalities of the form  
$\|x_n\|\leq 1$, for $n\in \bbN$,  are in the system. 
By compactness, we can assume each $K_n$ is a singleton $\{r_n\}$. 
Therefore we may assume 
\eqref{I.approx} consists of conditions of the form $|\| P_n(\bar x)\|-r_n|\leq 1/m$, 
for all $m$ and $n$. 
By re-enumerating $P_n$'s and adding redundancies, we may also assume that only the variables $x_j$, for $j\leq n$, occur in $P_n$ for every $n$. 
For each $m$ fix an approximate solution $\dot x_j(m)=\pi(x_j(m))$, for $j\leq m$, as in \eqref{I.approx}. 
Therefore 
\begin{enumerate}
\popcounter
\item \label{I.approx.1.1} $|\|P_k(\pi( \bar x(m)))\|-r_k |\leq 1/m$ for all $k\leq m$. 
\pushcounter
\end{enumerate}
We  choose all $x_k(m)$ to have norm $\leq 1$.

Let $P_n^0(\bar x)$ be a polynomial with coefficients in $M(A)$ that lift to the 
corresponding coefficients of $P_n(\bar x)$. 
Let $F_n$ be a finite subset of $M(A)$ such that $\pi(F_n)$ includes the following: 
\begin{enumerate}
\item [(i)] all coefficients
of every $P_j^0$ for  $j\leq n$, 
\item [(ii)]  $\{x_k(m):k\leq m\}$ satisfying \eqref{I.approx.1} for all $m\leq n$, 
and 
\item [(iii)] $\{P_j^0(x_0(j), \dots, x_j(j)): j\leq n\}$. 
\end{enumerate}
  With $\e_n=2^{-n}$ 
let $(e_n)$ and $(f_n)$ be as guaranteed by Lemma~\ref{L1}. 
Since $\|x_j(i)\|\leq 1$, by \eqref{I.bj} we have that 
\[
\textstyle y_i=\sum_j f_j x_j(i) f_j
\]
belongs to $M(A)$ for all $i$, and  \eqref{I.inX} implies $y_i\in X_{(f_n)}$
for all $i$. 

We shall prove $\|P_n(\pi (\bar y))\|=r_n$ for all $n$. 

By \eqref{I.sum1} we have $\|y_i\|\leq 2$.  
Fix $n$ and 
 a monomial $a x_k b$ of $P_n^0(\bar x)$. Then for all $j\geq n$ we have 
\[
\|a  f_j x_k(j) f_j b- f_j a x_k(j) b f_j\|\leq 2\e_j 
\]
and therefore the sum of these differences  is a convergent series in $A$ and we have
\begin{enumerate}
\popcounter
\item \label{I.approx.1} $ a (\sum_j f_j x_k(j) f_j) b \sim_A  \sum_j (f_j a x_k(j) b f_j)$. 
\pushcounter
\end{enumerate}
Since the polynomial 
$P_n^0(\bar x)$ has degree 1, all of its nonconstant monomials are either of the  form $a x_k b$  
or of the form $a x_k^* b$ for some $k$, $a$ and $b$, and by \eqref{I.approx.1}  
(writing $\sum_j f_j \bar y(j) f_j$ for the $n+1$-tuple
$(\sum_j f_j y_0(j) f_j, \dots, \sum_j f_j y_n (j) f_j)$)
\[
\textstyle P_n^0(\sum_j f_j y_k(j) f_j) \sim_A  \sum_j  f_j P_n^0( \bar y(j)) f_j. 
\]
By \eqref{I.binA} we have $\sum_j f_j y_i(j) f_j \sim_A \sum_j f_j y_i f_j$ for all $i$ 
and therefore  
\begin{equation*}
\textstyle P_n^0(\bar y)\sim_A P_n^0(\sum_j f_j \bar y(j) f_j)
\sim_A \sum_j f_j P_n^0(\bar y (j)) f_j. 
\end{equation*}
Using this, by  \eqref{I.sum2} 
we have that
\[
\|P_n(\bar y)\|=\|\pi(P_n^0(\bar y))\|=\limsup_j \| f_j P_n^0(\bar y) f_j\|
=r_n.
\]
Therefore $\pi(\bar y)$ is a solution to the system. 
Since the inequality $\| x_k\|\leq 1$ was in the system for all $k$ we also have
$\|y_k\|\leq 1$ for all $k$  and this concludes the proof. 
\end{proof}

\section{Limiting examples} 
\label{S.Limiting}
In this section we prove that the Calkin algebra is not countably saturated (cf. \cite{FaHaSh:Model2}). 
More precisely, in Proposition~\ref{P1} we construct a consistent type consisting of universal 
formulas that is not realized in the Calkin algebra. 
In Proposition~\ref{P2} we go a step further and present a proof, due to N. Christopher Phillips, 
 that some consistent quantifier-free type is not realized in the Calkin algebra.

For a unitary $u$ in a C*-algebra $A$ let 
\[
\xi(u)=\{j\in \bbN\mid u\text{ has a
$j$-th root}\}.
\]
 By Atkinson's theorem, every invertible operator in the Calkin algebra 
is the image of a Fredholm operator in $\cB(H)$ and therefore 
$\xi(u)$ is either $\bbN$ or $\{j\mid
j$ divides $m\}$ for some $m\in \bbN$, depending on whether the Fredholm index
of $u$ is 0 or $\pm m$.

Recall that a \emph{supernatural number} is a formal expression of the form 
$\prod_i p_i^{k_i}$, where $\{p_i\}$ is the enumeration of primes and each $k_i$ is a natural 
number (possibly zero) or $\infty$. The divisibility relation on supernatural numbers is defined 
in the natural way. 

\begin{prop}\label{P1}  
For any supernatural number $n$ the 
type $\bt(n)$ consisting of following conditions is finitely approximately realizable, 
but not realizable,  in the 
Calkin algebra. 
\begin{enumerate}
\item $x_0x_0^*=1$, $x_0^*x_0=1$, 
\item $x_k^k=x_0$, whenever $k$ is a natural number that divides $n$, 
\item $\inf_{\|y\|=1} \|y^k-x_0\|\geq 1$, whenever $k$ is a natural number that does not divide $n$.  
\end{enumerate}
In particular, the Calkin algebra is not countably saturated. 
\end{prop}

\begin{proof}
 We have  $n=\prod_j p_j^{k_j}$, where $(p_j)$ is the increasing
enumeration of primes and $k_j\in \bbN\cup \{\infty\}$. 

Let $s$ denote the unilateral shift on the underlying Hilbert space $H$ and let $\dot s$ be its
image in the Calkin algebra. 
For $l\in \bbN$ 
let $n_l=\prod_{j=1}^l p_j^{\min(k_j,l)}$. We claim that   
 \[
 \xi(\dot s^{n_l})=\{m\in \bbN\mid m\text{ divides }n\}.
\]
The inclusion is trivial. 
 In order to prove the converse inclusion  
fix $k\in \bbN$ that does not divide $n_l$.  
Assume for a moment that  $\dot s^{n_l}$ has a $k$-th root $\dot v$ in $\cC(H)^{\cU}$.  
Let  $u$ and $w$ be elements of $\cB(H)$ mapped to $\dot s^{n_l}$ and $\dot v^k$
by the quotient map. Then they are Fredholm operators with different Fredholm indices and
$\|\pi(u)\|=\|\pi(w)\|=1$.   Essentially by \cite[3.3.18 and 3.3.20]{Pede:Analysis} we have
 $\|\pi(u-w)\|\geq 1$, and therefore $v=\pi(w)$ is not $k$-th root of $\dot s^{n_l}$.  
\end{proof}

Proposition~\ref{P2} below was communicated to us   by N. Christopher Phillips in \cite{Phi:Email260611}. 
We would like to thank Chris for his kind permission to include this result here. 
While the proof in \cite{Phi:Email260611} relied entirely on known results 
about Pext and a topology 
on Ext  (more precisely, 
\cite[\S 3]{Sali:Relative}, \cite[\S 2]{Arv:Notes}, \cite[[Theorem 3.3]{Scho:FineII}, and 
\cite[Proposition 9.3 (1)]{Scho:Pext}), for convenience of the reader we shall present a self-contained proof of this result.

\begin{prop}\label{P2} There is a countable 
degree-1 type over the Calkin algebra that is finitely approximately
realizable by unitaries but not realizable by a unitary. 
In particular, the Calkin algebra is not countably quantifier-free saturated. 
\end{prop} 

\begin{proof} We include more details than a C*-algebraist may want to see. 
Recall that for a C*-algebra $A$ the abelian semigroup $\Ext(A)$ is defined as follows:
On the set of *-homomorphisms $\pi\colon A\to \cC(H)$ consider the conjugacy relation
by unitaries in $\cC(H)$. On the set of conjugacy classes define addition by 
letting $\pi_1\oplus \pi_2$ be the direct sum, where $\cC(H)$ is identified with $\cC(H\oplus H)$. 
The only fact about $\Ext$ that we shall need is that there exists a simple separable C*-algebra $A$ 
such that $A$ is a direct limit of algebras whose $\Ext$ is trivial, but $\Ext(A)$ is not trivial. 
For example, the CAR algebra has this property and we shall sketch a proof of this well-known 
fact below. 

Now fix $A$ as above and let $\pi_1\colon A\to \cC(H)$ and $\pi_2\colon A\to \cC(H)$ be 
inequivalent *-homomorphisms. Since $A$ is simple both  $\pi_1$ and $\pi_2$ are injective 
and  $F(\pi_1(a))=\pi_2(a)$  defines a map $F$ from $\pi_1[A]$ to $\pi_2[A]$. 
This map is not implemented by a unitary, but if $A=\lim_n A_n$ so that $\Ext(A_n)$ is trivial for every $n$, 
then the restriction of $F$ to  $\pi_1[A_n]$ is implemented by a unitary. 
Fix a countable dense subset $D$ of $\pi_1[A]$. Then the countable degree-1 type $\bt$ consisting of  
all  conditions of the form $xa=F(a)x$, for $x\in D$,  is finitely approximately 
realizable by a unitary, but not realizable by a unitary. 

We now sketch a proof that $\Ext$ of the CAR algebra $A=\bigotimes_n M_2(\bbC)$ is nontrivial. 
Write $A$ as a direct limit of $M_{2^n}(\bbC)$ for $n\in \bbN$. 
While $\Ext(M_{2^n}(\bbC))$ is trivial, the so-called \emph{strong $\Ext$} of $M_{2^n}(\bbC)$ 
is not. Two *-homomorphisms of $M_{2^n}(\bbC)$ into $\cC(H)$ are \emph{strongly equivalent} if 
they are conjugate by $\dot u$, for a unitary $u\in \cB(H)$. 
Every unital *-homomorphism $\Phi$ of $M_{2^n}(\bbC)$ into $\cC(H)$ 
is lifted by a *-homomorphism $\Phi_0$ into $\cB(H)$ and the 
strong equivalence class of $\Phi$ is uniquely determined by the 
 codimension of $\Phi_0(1)$ modulo $2^n$. Any unitary $u$ in $\cC(H)$ that witnesses 
 such $\Phi$ is conjugate to the trivial representation of $M_{2^n}(\bbC)$ which necessarily 
 has Fredholm index equal to the codimension of $\Phi_0(1)$ modulo $2^n$. 
 Now write $M_{2^\infty}$ as $\bigotimes_\bbN A_n$ where $A_n\cong M_2(\bbC)$ for all $n$. 
 Recursively find *-homomorphisms $\pi_1^n$ and $\pi_2^n$ from  $\bigotimes_{j\leq n} A_j$
 into the Calkin algebra 
 so that (i) $\pi_j^{n+1}$ extends $\pi_j^n$ for all $n$ and $j=1,2$, (ii) each $\pi_1^n$
 has trivial strong $\Ext$ class, and (iii) each $\pi_2^n$ has strong $\Ext$ class 
$2^{n-1}$  (modulo $2^n$). The construction is straightforward. 
The limits $\pi_1$ and $\pi_2$ are *-homomorphisms of the CAR algebra into the Calkin algebra
such that the first one lifts to a homomorphism of the CAR algebra into $\cB(H)$ and the 
other one does not. 
\end{proof}

\section{Concluding remarks} 
\label{S.Concluding} 

Both obstructions to the countable saturation of the Calkin algebra described in \S\ref{S.Limiting}
have a K-theoretic nature. 

\begin{question} \label{Q.K} Do all obstructions to countable  saturation, or at least 
to countable quantifier-free saturation,  of corona algebras have a
K-theoretic nature? 
\end{question} 

A test question for Question~\ref{Q.K} was suggested by It\"ai Ben Ya'acov. 
Consider the unitary group $\cU(\cC(H))$ of the Calkin algebra with respect to the 
gauge given by the Fredholm index (see \cite{BY:Continuous}). Is this structure quantifier-free 
countably saturated? 
Here is an even less ambitious test question:

\begin{question} Let $\cU_0$ be the subgroup of the unitary group of the Calkin algebra consisting
of unitaries of Fredholm index zero. Is this structure quantifier-free countably saturated in 
the logic of metric structures? 
\end{question}

A discrete total ordering $L$ is countably saturated (in the classical model-theoretic 
sense, see e.g., \cite{ChaKe}) if and only if whenever $X$ and $Y$ are countable 
subsets  of $L$ such that $x<y$ for all $x\in X$ and all $y\in Y$ and either 
$X$ has no maximal element or $Y$ has no minimal element  there
is $z\in L$ such that $x<z$ and $z<y$ for all $x\in X$ and all $y\in Y$. 
The fact that this definition does not involve formulas of arbitrary complexity is a consequence of 
the classical result that theory of dense linear orderings allows elimination of quantifiers. 

Hadwin proved (\cite{Had:Maximal}) that every maximal chain of projections in the Calkin 
algebra, when considered
as a discrete linear ordering, is countably saturated. 
It should be noted that not every maximal commuting family of projections in the Calkin algebra 
is countably saturated. For example, the family of projections of an atomless masa 
is isomorphic to the Lebesgue measure algebra. The latter is a  complete Boolean algebra 
and therefore not countably saturated. 
It is not difficult to see that in every quantifier-free countably saturated algebra
every maximal chain of projections is countably saturated. 
However,  the Calkin algebra is not countably quantifier-free saturatedÊ(Proposition~\ref{P2}) 
 and we don't know whether countable degree-1 saturation suffices for Hadwin's result. 
 An affirmative answer to Question~\ref{Q.proj} below would suffice for this.

We also don't know whether in a countably degree-1 saturated C*-algebra 
every countable type that is approximately  finitely satisfiable 
by projections is realized by a projection (cf. Lemma~\ref{L.sa} and Lemma~\ref{L.un}). 
An affirmative answer would imply an affirmative answer to the following:
(see Definition~\ref{D.gap}).  

\begin{question} \label{Q.proj} 
Assume $C$ is a countably degree-1 saturated C*-algebra. 
If $A$ and $B$ are separable orthogonal subalgebras  of $C$, 
are they necessarily separated by a projection? 
\end{question}

A positive answer to the following would provide a more satisfying proof of Corollary~\ref{C.Mn}. 

\begin{question} \label{Q.Mn}
Assume $C$ is countably degree-1 saturated. Is $M_n(C)$ countably degree-1 saturated
for all $n$? 
\end{question}

The statement of \cite[Proposition~9.1]{Pede:Corona} distinguishes between 
$C$, $M_2(C)$ and $M_4(C)$ being SAW*-algebras and therefore 
suggests that the analogous assertion for SAW*-algebras
is false, or at least not obviously true. A closely related problem to Question~\ref{Q.Mn}
is the following (the relevant definition of `definable' is as in \cite{BYBHU} or \cite{FaHaSh:Model2}): 

\begin{question} \label{Q.def} If $A$ is a C*-algebra and $n\geq 2$, is the unit ball of $M_n(A)$, 
when identified with a subset of $(A_{\leq 1})^{n^2}$, definable over $A$? 
If so, what is  the logical complexity of the definition? 
\end{question}

A positive answer to Question~\ref{Q.def} would imply that if $C$ is countably saturated
then so is $M_n(C)$. 
By \cite[Proposition~9.1]{Pede:Corona}, this would imply, for example, that 
countably degree-1 saturated C*-algebras allow weak polar decomposition.

Following \cite{CoFa:Automorphisms} we say that an automorphism $\Phi$ of $M(A)/A$
is \emph{trivial} if the set $\{(a,b)\in M(A)^2: \Phi(a/A)=b/A\}$ is strictly Borel. 
Every inner automorphism is clearly of this form.  Also in the case when $A$ is separable $M(A)$ 
is separable metric in the strict topology and therefore 
$M(A)/A$ has at most $2^{\aleph_0}$ trivial automorphisms. 

\begin{problem} \label{P.CH} Prove that the Continuum Hypothesis implies that every corona of an infinite-dimensional, non-unital, 
separable C*-algebra has nontrivial  automorphisms. 
\end{problem}

A positive answer to this problem for a large class of C*-algebras, including 
all stable C*-algebras of real rank zero, is given in \cite{CoFa:Automorphisms}. Methods
of the present paper of \cite{CoFa:Automorphisms} do not apply to the algebra  $C([0,1))$. 
However, Problem~\ref{P.CH} is known to have an affirmative solution  by a result of 
J. C. Yu (see \cite[\S 9]{Hart:Cech}). We don't know the degree of saturation of the  
corona of $C([0,1))$. 
An even more interesting problem  is to prove that an appropriate forcing axiom implies 
all automorphisms of all coronas of separable C*-algebras are trivial (this is necessarily 
 weaker than  `inner,' see the last section of \cite{Fa:All}). 

A C*-algebra $C$ is \emph{countably homogeneous} (as a metric structure) if 
for every two sequences $\langle a_n: n\in \bbN\rangle$ and $\langle b_n: n\in \bbN\rangle$ 
in $C$ that have the same type in $C$ (see \cite{FaHaSh:Model2}) there exists an 
automorphism $\Phi$ of $C$ such that $\Phi(a_n)=b_n$ for all $n$. 
Fully countably saturated C*-algebras of character density $\aleph_1$ are 
 are countably homogeneous. Therefore for example the Continuum Hypothes isimplies that 
 ultrapowers of separable C*-algebras associated with nonprincipal ultrafilters on $\bbN$ are 
 countably homogeneous. 

\begin{question} Are corona algebras of $\sigma$-unital C*-algebras countably homogeneous? 
\end{question}

A positive answer in the case of the Calkin algebra $\cC(H)$ would imply that 
the unitaleral shift and its adjoint have the same type if and only if 
it is relatively consistent with ZFC that $\cC(H)$ has a K-theory reversing automorphism. 

\providecommand{\bysame}{\leavevmode\hbox to3em{\hrulefill}\thinspace}
\providecommand{\MR}{\relax\ifhmode\unskip\space\fi MR }
\providecommand{\MRhref}[2]{%
  \href{http://www.ams.org/mathscinet-getitem?mr=#1}{#2}
}
\providecommand{\href}[2]{#2}


\begin{thebibliography}{10}

\bibitem{AkPed}
C.~A. Akemann and G.~K. Pedersen, \emph{Central sequences and inner derivations
  of separable ${C}^*$-algebras}, Amer. J. Math. \textbf{101} (1979),
  1047--1061.

\bibitem{Arv:Notes}
W.~Arveson, \emph{Notes on extensions of {C*}-algebras}, Duke Math. J.
  \textbf{44} (1977), 329--355.

\bibitem{BY:Continuous}
I.~Ben~Yaacov, \emph{Continuous first order logic for unbounded metric
  structures}, Journal of Mathematical Logic \textbf{8} (2008), 197--223.

\bibitem{BYBHU}
I.~Ben~Yaacov, A.~Berenstein, C.W. Henson, and A.~Usvyatsov, \emph{Model theory
  for metric structures}, Model Theory with Applications to Algebra and
  Analysis, Vol. II (Z.~Chatzidakis et~al., eds.), London Math. Soc. Lecture
  Notes Series, no. 350, Cambridge University Press, 2008, pp.~315--427.

\bibitem{Black:Operator}
B.~Blackadar, \emph{Operator algebras}, Encyclopaedia of Mathematical Sciences,
  vol. 122, Springer-Verlag, Berlin, 2006, Theory of $C\sp *$-algebras and von
  Neumann algebras, Operator Algebras and Non-commutative Geometry, III.

\bibitem{ChaKe}
C.~C. Chang and H.~J. Keisler, \emph{Model theory}, third ed., Studies in Logic
  and the Foundations of Mathematics, vol.~73, North-Holland Publishing Co.,
  Amsterdam, 1990.

\bibitem{CoFa:Automorphisms}
S.~Coskey and I.~Farah, \emph{Automorphisms of corona algebras and group
  cohomology}, preprint, arXiv:1204.4839, 2012.

\bibitem{Fa:All}
I.~Farah, \emph{All automorphisms of the {C}alkin algebra are inner}, Annals of
  Mathematics \textbf{173} (2011), 619--661.

\bibitem{FaHaSh:Model2}
I.~Farah, B.~Hart, and D.~Sherman, \emph{Model theory of operator algebras
  {II}: Model theory}, preprint, arXiv:1004.0741, 2010.

\bibitem{FaHaSh:Model1}
\bysame, \emph{Model theory of operator algebras {I}: Stability}, Proc. London
  Math. Soc. (to appear).

\bibitem{Gha:SAW*}
S.~Ghasemi, \emph{{SAW}* algebras and tensor products}, preprint,
  arXiv:1209.3459, 2012.

\bibitem{Had:Maximal}
D.~Hadwin, \emph{Maximal nests in the {C}alkin algebra}, Proc. Amer. Math. Soc.
  \textbf{126} (1998), 1109--1113.

\bibitem{Hart:Cech}
Klaas~Pieter Hart, \emph{The \v {C}ech-{S}tone compactification of the real
  line}, Recent progress in general topology ({P}rague, 1991), North-Holland,
  Amsterdam, 1992, pp.~317--352.

\bibitem{Hig:Technical}
Nigel Higson, \emph{On a technical theorem of {K}asparov}, J. Funct. Anal.
  \textbf{73} (1987), no.~1, 107--112.

\bibitem{Kirc:Central}
E.~Kirchberg, \emph{Central sequences in {$C\sp *$}-algebras and strongly
  purely infinite algebras}, Operator Algebras: The Abel Symposium 2004, Abel
  Symp., vol.~1, Springer, Berlin, 2006, pp.~175--231.

\bibitem{O:findim}
T.~Ogasawara, \emph{Finite-dimensionality of certain {B}anach algebras}, J.
  Sci. Hiroshima Univ. Ser. A. \textbf{17} (1954), 359--364.

\bibitem{Pede:C*}
Gert~K. Pedersen, \emph{{$C\sp{\ast} $}-algebras and their automorphism
  groups}, London Mathematical Society Monographs, vol.~14, Academic Press Inc.
  [Harcourt Brace Jovanovich Publishers], London, 1979.

\bibitem{Pede:Corona}
\bysame, \emph{The corona construction}, Operator {T}heory: {P}roceedings of
  the 1988 {GPOTS}-{W}abash {C}onference ({I}ndianapolis, {IN}, 1988), Pitman
  Res. Notes Math. Ser., vol. 225, Longman Sci. Tech., Harlow, 1990,
  pp.~49--92.

\bibitem{Pede:Analysis}
G.K. Pedersen, \emph{Analysis now}, Graduate Texts in Mathematics, vol. 118,
  Springer-Verlag, New York, 1989.

\bibitem{Phi:Email260611}
N.C. Phillips, An email to {I}lijas {F}arah, June 26, 2011.

\bibitem{PhWe:Calkin}
N.C. Phillips and N.~Weaver, \emph{The {C}alkin algebra has outer
  automorphisms}, Duke Math. Journal \textbf{139} (2007), 185--202.

\bibitem{Sali:Relative}
N.~Salinas, \emph{Relative quasidiagonality and {KK}-theory}, Houston J. Math.
  \textbf{18} (1992), 97--116.

\bibitem{Scho:FineII}
Claude~L. Schochet, \emph{The fine structure of the {K}asparov groups. {II}.
  {T}opologizing the {UCT}}, J. Funct. Anal. (2002), 263--287.

\bibitem{Scho:Pext}
\bysame, \emph{A {P}ext primer: {P}ure extensions and lim1 [$\lim^1$] for
  infinite abelian groups}, vol.~1, New York Journal of Mathematics, NYJM
  Monographs, 2003.

\end{thebibliography}
\end{document}